\documentclass[11pt]{article}

\usepackage[utf8]{inputenc}

\usepackage{geometry}
\usepackage{graphicx}

\usepackage{booktabs} 
\usepackage{array} 
\usepackage{paralist} 
\usepackage{verbatim} 
\usepackage{subfig} 
\usepackage{tabularx}
\usepackage{fancyhdr} 
\usepackage[nottoc,notlof,notlot]{tocbibind}
\usepackage[titles,subfigure]{tocloft}

\usepackage{amsmath,amsfonts,amsthm,mathrsfs,amssymb,cite}

\newcommand{\R}{{\mathbb R}}

\newcommand{\be}{\begin{eqnarray}}
\newcommand{\ben}{\begin{eqnarray*}}
\newcommand{\en}{\end{eqnarray}}
\newcommand{\enn}{\end{eqnarray*}}
\newcommand{\ba}{\backslash}
\newcommand{\pa}{\partial}

\newcommand{\ov}{\overline}
\newcommand{\curl}{{\rm curl\,}}

\newcommand{\g}{\gamma}
\newcommand{\G}{\Gamma}

\newcommand{\Om}{\Omega}

\newcommand{\La}{\Lambda}

\newtheorem{thm}{Theorem}[section]
\newtheorem{cor}{Corollary}[section]
\newtheorem{lem}{Lemma}[section]

\theoremstyle{definition}
\newtheorem{defn}{Definition}[section]

\newtheorem{rem}{Remark}[section]
\numberwithin{equation}{section}
\usepackage{graphicx}

\title{\bf Uniqueness to Inverse Acoustic and Electromagnetic Scattering From Locally Perturbed Rough Surfaces}
\author{
Yu Zhao \thanks{Shandong Normal University, Jinan 250358, China. Email: {\tt 2016020556@stu.sdnu.edu.cn}},
Guanghui Hu \thanks{Beijing Computational Science Research Center, Beijing 100193, China. Email: {\tt hu@csrc.ac.cn}},
Baoqiang Yan \thanks{Shandong Normal University, Jinan 250358, China. Email: {\tt yanbqcn@aliyun.com}}}

\date{} 
\begin{document}
\maketitle

\begin{abstract}
In this paper, we consider  inverse time-harmonic acoustic and electromagnetic scattering from locally perturbed rough surfaces in three dimensions. The scattering interface is supposed to be the graph of a Lipschitz continuous function with compact support. It is proved that an acoustically sound-soft or sound-hard surface can be uniquely determined by the far-field pattern of infinite number of incident plane waves with distinct directions. Moreover, a single point source or plane wave can be used to uniquely determine a scattering surface of polyhedral type. These uniqueness results apply to Maxwell equations with the perfectly conducting boundary condition. Our arguments rely on the mixed reciprocity relation in a half space and the reflection principle for Helmholtz and Maxwell equations.
\end{abstract}
\section{Introduction}\label{sec:intro}
The problem of half-space scattering with local perturbations has attracted widespread attention in the past two decades. Such kind of scattering problems has practical applications in radar, sonar, ocean surface detection, medical detection and so on. This paper is concerned with the time-harmonic acoustic scattering from locally perturbed rough surfaces; see  Figure \ref{fig:1} for an illustration the scattering problem in 2D where $\Gamma$ is the entire scattering interface and $\Lambda_{R}$ is a local perturbation of the unperturbed flat surface (ground plane). When the local perturbation lies  below the ground plane, it is also known as cavity scattering problems \cite{ABW2000,ABW2002,BGL2011}. In the literature, the problem considered in our context is sometimes referred to as the overfilled cavity scattering problem \cite{LWZ, Wood2006,ZZ2013}.

The forward acoustic scattering problems have been investigated intensively using integral equation or variational methods; see \cite{ABW2000, ABW2002, BGL2011, Wood2006, ZZ2013}. In the time-harmonic regime, it is well-known that the total field can be decomposed into three part: the incoming wave $u^{in}$, the reflected wave $u^{re}$ corresponding to the unperturbed scattering interface and the scattered wave $u^{sc}$ caused by the presence of local perturbations. Under the Sommerfeld radiation condition of $u^{sc}$ in the half space, one can show uniqueness and existence if the total field fulfills the Dirichlet or Neumann boundary condition, whereas the reflected waves are usually uniquely determined by the Snell's law in physics. We refer to \cite{Lip2018,CMG2017} for the mathematical modeling and analysis of the open cavity scattering problems in both the time-harmonic and time-dependent regimes.
\begin{figure}[ht]
\centering
\includegraphics[scale=0.6]{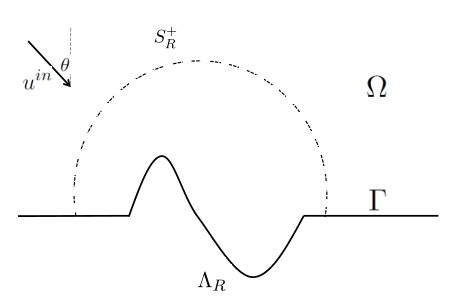}
\caption{Illustration of our scattering problem in a half plane with local perturbations in two dimensions.}
\label{fig:1}
\end{figure}

This manuscript is concerned with uniqueness to inverse scattering problems in a locally perturbed half-space.
For bounded obstacles, the first uniqueness result to inverse scattering from a sound-soft obstacle was given by Schiffer (see \cite{LPR1967}) with infinite number of plane waves with distinct incoming directions. It was shown in \cite{CS1983} by Colton and Sleeman that one plane wave is sufficient, provided a priori information on the size of the obstacle is available. Isakov \cite{Isa1990, Isa2006} proved the same uniqueness result for penetrable scatterers through the idea of  blow-up of point sources, which was later simplified by Kirsch and Kress in \cite{KK1993} and extended to other boundary conditions for impenetrable obstacles. Using the mixed reciprocity relation (see Potthast \cite{Pot2001}),  Isakov's proof can be further simplified via the reciprocity relation but still requires infinite number of incoming waves; see \cite[Theorem 5.6]{CK2013} in the acoustic case and \cite{K2002} in the Maxwell case.

In a locally perturbed half space,
Feng and Ma \cite{FM2005} show a uniqueness result with all incoming plane waves from above based on Schiffer's proof for sound-soft obstacles. The key point is to prove the linear independence of the total fields in a half space excited by distinct incident plane waves but at the same wavenumber. This would contradict the fact there are only finitely many linearly independent Dirichlet eigenfunctions in the difference domain of two sound-soft interfaces generating the same data. The first aim of this paper is to simply the proof of \cite{FM2005} by establishing the mixed reciprocity relation in a half space and then to prove uniqueness under the Neumann boundary condition.
The second aim of this paper is to prove uniqueness with a single incoming wave for polyhedral surfaces through the reflection and path arguments (see e.g., \cite{CY2003,DL1955,AR2005,EY2006,LZ2006}). Note that for bounded acoustic obstacles, uniqueness with a single plane wave was derived in \cite{CY2003,AR2005,LZ2006} for sound-soft scatterers and \cite{EY2006} for sound-hard scatterers of polyhedral type, and uniqueness with a single point source wave was deduced in \cite{HL2014}.


To the best of our knowledge, the analogue of Theorem \ref{alldirection} for general rough surfaces is not available in the literature, because the substraction of the incident plane wave from the total field does not fulfill the Sommerfeld radiation condition in general. Even using near-field data, it seems unknown how to uniquely determine a periodic sound soft surface with incoming plane waves excited at different incident directions from above.
The uniqueness results of Theorems \ref{alldirection} and \ref{harduni} carry over to impenetrable surfaces of impedance type (that is, Robin boundary condition) straightforwardly, but it is still open how to determine polyhedral surfaces of impedance type with a single incoming wave, even for bounded obstacles of impedance type.
We refer to \cite{BHY2018} for the well-posedness of forward scattering problems and the unique determination of a rectangular cavity with a single point wave under the impedance boundary condition.


The outline of the article is organized as follows. In Section \ref{sec:direct}, we introduce the mathematical model of acoustic scattering problems in a locally perturbed half space. Section \ref{sec:soft} is devoted to the unique determination of a sound-soft surface with one or multiple incident directions, while Section \ref{sec:hard} presents the uniqueness results for recovering sound-hard surfaces. The electromagnetic scattering problems will be discussed in Section \ref{elec}.

\section{Acoustic scattering problem}\label{acoustic}

\subsection{Mathematical settings}\label{sec:direct}
Let $\Gamma:=\{(\tilde{x}, x_{3}), x_{3}=f(\tilde{x})\} \subset \mathbb{R}^{3}$ be a local perturbation of the ground plane $x_{3}=0$, where $\tilde{x}:=(x_1,x_2)$ and $f$ is a Lipschitz continuous function on $\mathbb{R}^{2}$. It is supposed that $f(\tilde{x})=0$ for $|\tilde{x}|>R$, where $R>\max\{|f(\tilde{x})|:\tilde{x}\in\Gamma\}$. Obviously, $\Gamma$ consists of two parts: the flat part $\Gamma\cap\{x:|x|>R\}$ and the perturbed part $\Lambda_{R}:=\Gamma\cap\{(\tilde{x},x_{3}),|x|<R\}$. Denote by $\Omega:=\{(\tilde{x},x_{3}),x_{3}>f(\tilde{x})\}$ the region above $\Gamma$, which is supposed to be filled with an infinite isotropic and homogeneous acoustic medium. The incident field is given by the time-harmonic acoustic plane wave
$$U^{in}(x;t)=u^{in}(x)e^{-i\omega t}=e^{i(kx\cdot d-\omega t)},$$
where $k=\omega/{c_{0}}>0$ is the wave number, $\omega$ is the frequency, $c_{0}$ is the speed of sound and $d$ is the direction of propagation. Assume that $u^{in}$ is incident onto $\Gamma$ from $\Omega$. The wave propagation can be modeled by the Helmholtz equation
\be\label{2.1}
\triangle u+k^{2}u=0~~~~\mbox{in}~\Omega,
\en
where $u$ denotes the total field. If the medium below $\Gamma$ is sound soft, there holds the Dirichlet boundary condition $u=0$ on $\Gamma$. If it is sound hard, we have the Neumann boundary condition $\partial_{\nu}u=0$ on $\Gamma$, where $\nu\in \mathbb{S}^{2}:=\{x\in\mathbb{R}^3, |x|=1\}$ is the unit normal vector at $\Gamma$ directed into $\Omega$. It is well known from well-posedness of forward scattering problems (see e.g. \cite{Lip2018}) that the total field includes three parts: the incoming wave $u^{in}$, the scattered wave $u^{sc}$ and the reflected wave $u^{re}$. The scattered wave $u^{sc}$ is caused by the perturbed part and $u^{re}$ corresponds to the unperturbed scattering interface $\{x_{3}=0\}$. Here, the sum $\widetilde{u}^{in}:=u^{in}+u^{re}$ solves the unperturbed scattering problem that corresponds to the flat surface $\{x_3=0\}$. Then, the total field can be rewritten as
$$u=\widetilde{u}^{in}+u^{sc}~~~~\mbox{in} ~\Omega.$$
Without loss of generality, we suppose that the incident direction is of the form
 $d=(\alpha, \beta, -\gamma)\in \mathbb{S}^{2}_{-}:=\{x\in \mathbb{S}^{2},x_{3}<0\}$, so that the incident field takes the form
$$u^{in}(x;d)=e^{ikx\cdot d}=e^{ik(\alpha x_{1}+\beta x_{2}-\gamma x_{3})},$$
where $\alpha=\sin \varphi \cos\theta , ~\beta=\sin \varphi \sin \theta, ~\gamma= \cos \varphi$ with the incident angles $\varphi \in (-\frac{\pi}{2},\frac{\pi}{2}),~\theta \in (0,2\pi)$. By Snell's law, the reflected wave of the plane wave $u^{in}(x;d)$ incident onto $\{x_3=0\}$ takes the explicit form
\[
u^{re}(x;d)=
\left\{\begin{array}{l}
-e^{ik(\alpha x_{1}+\beta x_{2}+\gamma x_{3})}~~~~\mbox{in the Dirichlet case},\\
e^{ik(\alpha x_{1}+\beta x_{2}+\gamma x_{3})}~~~~~~\mbox{in the Neumann case}.
\end{array}
\begin{array}{l}
\end{array}\right.
\]
Let $B^{+}_{R}:=\{|x|<R: x_{3}>0\}$ be the upper half-ball containing the local perturbation $\Lambda_{R}$ and let $S^{+}_{R}:=\{|x|=R: x_{3}>0\}$ be the boundary of $B^{+}_{R}$. In the upper half space, the scattered field $u^{sc}$ is required to satisfy the Sommerfeld radiation condition
\be\label{2.2}
\lim\limits_{r\rightarrow\infty}r\{\partial_{r} u^{sc}(x;d)-iku^{sc}(x;d)\}= 0,
\en
where $r=|x|$, $x\in \Omega$ and $d\in \mathbb{S}^{2}_{-}$. It is well known from \cite[Definition 2.4]{CK2013} that the Sommerfeld solution $u^{sc}$ admits the asymptotic behavior
\be\label{2.3}
u^{sc}(x;d)=\frac{e^{ik|x|}}{|x|}\left\{u_{\infty}^{sc}(\hat{x};d)+O\left(\frac{1}{|x|}\right)\right\}, ~|x|\rightarrow\infty,~x\in\Omega,
\en
where $\hat{x}:=\frac{x}{|x|}\in \mathbb{S}^{2}_{+}:=\{x\in \mathbb{S}^{2},x_{3}>0\}$ and $u_{\infty}^{sc}(\hat{x})$ is the far-field pattern of $u^{sc}$. Note that in \eqref{2.3}, we have emphasized the dependance of $u^{sc}$ and $u^{sc}_\infty$ on the incident direction $d$.

The direct scattering problem (DP) is stated as follows.
\begin{description}
\item[(\textbf{DP}):]Give an incoming plane wave $u^{in}$ and a locally perturbed surface $\Gamma \subset \mathbb{R}^{3}$, determine the total field $u$ in $\Omega$ under the Dirichlet or Neumann boundary condition.
\end{description}

We are interested in the following inverse problem (IP) arising from acoustic scattering.
\begin{description}
\item[\bf (IP):] Determine the shape of $\Gamma$ (more precisely, $\Lambda_R$) from knowledge of the far-field patterns $u_{\infty}^{sc}(\hat{x};d)$ for all observation directions $\hat{x}\in \mathbb{S}^2_{+}$ corresponding to one or many incident plane waves with directions $d\in \mathbb{S}^{2}_{-}$ and a fixed wave number.
\end{description}

This section is concerned with uniqueness to (IP) under the Dirichlet or Neumann boundary condition. We shall prove in Theorems \ref{alldirection} and \ref{harduni} that the data of all incident directions can be used to determine $\Lambda_R$ uniquely, while a single direction is sufficient if $f$ is piecewise linear. The electromagnetic scattering problem will be discussed in Section \ref{PolyRef}.

\subsection{Uniqueness in determining sound-soft surfaces}\label{sec:soft}
In the following theorem, we use the reflection principle for the Helmholtz equation (see e.g. \cite{DL1955}) to extend the scattered field to a lower half space and then apply
 Rellich's lemma (cf. \cite[Lemma 2.12]{CK2013}) to prove the one-to-one correspondence between the scattered field and its far field pattern in a half space.
\begin{thm}\label{farscat}
Assume that $\Gamma_{j}:=\{x:x_{3}=f_{j}(\tilde{x})\}(j=1,2)$ are two sound-soft surfaces with local perturbations such that their far-field patterns coincide for the incident plane wave with the direction $d\in \mathbb{S}^{2}_{-}$. Then, it holds that
$$u^{sc}_{1}(x;d)=u^{sc}_{2}(x;d),$$
for all $x\in \Omega_{1}\cap \Omega_{2}$, where $\Omega_{j}:=\{x\in\mathbb{R}^{3}: x_{3}>f_{j}(\tilde{x})\}, j=1,2$.
\end{thm}
\begin{proof}
Let $\Lambda_{R}^{(j)}:=\Gamma_{j}\cap\{|x|<R\} (j=1,2)$ be the perturbed part of $\Gamma_{j}$ for some $R>\max\{|f_j(\tilde{x})|:\tilde{x}\in\Gamma_{j}\}$. We use $u_{j}^{sc}$ to represent the scattered field of $\Gamma^{(j)}$ and let $u_{j,\infty}^{sc}(\hat{x};d)$ be the far-field pattern of $u_{j}^{sc}(x;d)$, where $x\in\Omega_j$, $\hat{x}\in \mathbb{S}_{+}^{2}$ and $d\in \mathbb{S}^{2}_{-}$.

Apparently, $u^{sc}_{1}$ satisfies the Helmholtz equation in $\Omega_{1}$. We can extend $u^{sc}_{1}$ from $\mathbb{R}^{3}_{+}\backslash \overline{B^{+}_{R}}$ to $\mathbb{R}^{3}\cap\{|x|>R\}$ by\\
\be\label{3.1}
 V^{sc}_{1}(\tilde{x},x_{3})=
\left\{\begin{array}{l}
u^{sc}_{1}(\tilde{x},x_{3}),~~~~~~~|x|>R, ~x_{3}>0,\\
u^{sc}_{1}(\tilde{x},0),~~~~~~~~~|x|>R,~x_{3}=0,\\
-u^{sc}_{1}(\tilde{x},-x_{3}),~~~|x|>R,~x_{3}<0.
\end{array}
\begin{array}{l}
\end{array}\right.
\en
It is easy to find
$$\triangle V^{sc}_{1}(x)+k^{2}V^{sc}_{1}(x)=0~~~~\mbox{in}~\{x:|x|>R, x_3\neq0\}.$$
Besides, we have
$$\lim\limits_{x_{3}\rightarrow 0^{+}} V^{sc}_{1}(x)=\lim\limits_{x_{3}\rightarrow 0^{-}} V^{sc}_{1}(x)$$
and
$$\lim\limits_{x_{3}\rightarrow 0^{+}} \partial_{x_3} V^{sc}_{1}(x)=\lim\limits_{x_{3}\rightarrow 0^{-}} \partial_{x_3} V^{sc}_{1}(x),$$
for all $|x|>R$. Applying \cite[Lemma 6.13]{K2011}, it can be shown that
$$\triangle V^{sc}_{1}(x) + k^{2}V^{sc}_{1}(x) = 0~\mbox{in}~\mathbb{R}^{3}\cap\{|x|>R\}.$$
Analogously, one can extend $u_{2}^{sc}$ by $V_{2}^{sc}$ from $\mathbb{R}^{3}_{+}\backslash \overline{B^{+}_{R}}$ to $\mathbb{R}^{3}\cap\{|x|>R\}$ in the same way as for $u_{1}^{sc}$.

By the definition of $V_{j}^{sc}$ and \eqref{2.3}, $V_{j}^{sc}(x;d)$ satisfies the Sommerfeld radiation condition in $|x|>R$. In fact, for $x_3<0$ we have
\ben
V_j^{sc}(x;d)=\frac{e^{ik|x|}}{|x|}\{-V_{j,\infty}^{sc}(\hat{x}';d)+O(\frac{1}{|x|})\}, ~|x|\rightarrow\infty,
\enn
where $\hat{x}'\in \mathbb{S}^{2}$ denotes the reflection of $\hat{x}$ about the $x_3$ axis.
Hence, we can define $V_{j,\infty}^{sc}(\hat{x};d)$ as the far-field pattern of $V_{j}^{sc}(x;d)$ for all $\hat{x}\in\mathbb{S}^2$. Moreover, we obtain $V_{j,\infty}^{sc}(\hat{x};d)=-V_{j,\infty}^{sc}(\hat{x}';d)$. Therefore, it follows from $u_{1,\infty}^{sc}(\hat{x};d)=u_{2,\infty}^{sc}(\hat{x};d)$, $\hat{x}\in\mathbb{S}^{2}_{+}$ that $V_{1,\infty}^{sc}(\hat{x};d)=V_{2,\infty}^{sc}(\hat{x};d)$ for all $\hat{x}\in\mathbb{S}^{2}$. By the asymptotic behavior of the scattered fields, we can get
\ben
W(x;d)&:=& V^{sc}_{1}(x;d)-V^{sc}_{2}(x;d)\\
&=&\frac{e^{ik|x|}}{|x|} \{V_{1,\infty}^{sc}(\hat{x};d)-V_{2,\infty}^{sc}(\hat{x};d)+O(\frac{1}{|x|})\}\\
&=&O(\frac{1}{|x|^{2}}),~~|x|\rightarrow\infty
\enn
which implies that
$$\lim\limits_{R\rightarrow\infty} \int_{S_{R}} |W(x;d)| ^{2} ds(x)=0.$$
On the other hand, $W(x;d)$ is a radiation solution to the Helmholtz equation in $|x|>R$. By Rellich's Lemma (cf. \cite[Lemma 2.12]{CK2013}), we can obtain $W(x;d)=0$ in $|x|>R$. Consequently, we have
$$u^{sc}_{1}(x;d)=u^{sc}_{2}(x;d)~~~~\mbox{in}~ \mathbb{R}^{3}_{+} \cap \{|x|>R\}.$$
Recalling the unique continuation principle for the Helmholtz equation (cf. \cite[Theorem 8.6]{CK2013}), we obtain
$$u^{sc}_{1}(x;d)=u^{sc}_{2}(x;d)~~~~\mbox{in}~\Omega_{1}\cap \Omega_{2}.$$
\end{proof}
\begin{rem}\label{remark 2.1}
The proof of Theorem \ref{farscat} extends to the Neumann boundary condition trivially. Further, we think that the impedance case can be treated analogously by applying the corresponding reflection principle (see \cite{DL1955}). However, it remains unclear to us how to handle the transmission interface conditions.
\end{rem}
\subsubsection{Uniqueness with multiple incident waves}\label{sec:unimult}
The aim of this section is to show that a sound-soft surface can be uniquely determined by multiple incident waves with distinct directions. Similar to plane waves, below we describe the reflected field for point source waves. The fundamental solution of the Helmholtz equation \eqref{2.1} is given by
$$\Phi (x,y):=\frac{1}{4\pi} \frac{e^{ik|x-y|}}{|x-y|},~~x\neq y.$$
For the scattering of the point source wave
$$w^{in}(x;z)=\Phi(x;z),~~x\neq z ,$$
where $z=(z_{1}, z_{2}, z_{3})\in\Omega$,
we denote the scattered field by $w^{sc}$, the reflected field by $w^{re}$, the total field by $w$ and the far-field pattern by $w^{sc}_{\infty}$. Note that $w^{sc}$ is caused by the local perturbation and $w^{re}$ is resulted from the unperturbed scattering interface. Similar to the plane wave case, the total wave can be written as $w=w^{in}+\widetilde{w}^{sc}$, where $\widetilde{w}^{sc}$ is analytic in $\Omega$ and is of the form
\[
\widetilde{w}^{sc}(x;z):=
\left\{\begin{array}{l}
w^{re}(x;z)+w^{sc}(x;z)~~~~\mbox{in}~\Omega\cap\{|x|>R\},\\
\widetilde{w}^{sc}(x;z)~~~~~~~~~~~~~~~~~~~\mbox{in}~\Omega\cap\{|x|<R\}.
\end{array}
\begin{array}{l}
\end{array}\right.
\]
By the definition of point source wave and the Snell's law, we obtain
\[
w^{re}(x;z)=
\left\{\begin{array}{l}
-\frac{1}{4\pi} \frac{e^{ik|x-z'|}}{|x-z^{'}|},~~x\neq z',~~~~\mbox{in the Dirichlet case,}\\
\frac{1}{4\pi} \frac{e^{ik|x-z'|}}{|x-z^{'}|},~~x\neq z',~~~~~~\mbox{in the Neumann case,}
\end{array}
\begin{array}{l}
\end{array}\right.
\]
where $z':=(z_{1}, z_{2}, -z_{3})$.
Obviously, the scattered field of the point source wave can be written as
\[
\ w^{sc}(x;z):=
\left\{\begin{array}{l}
w^{sc}(x;z)~~~~~~~~~~~~~~~~~~~\mbox{in}~\Omega\cap\{|x|>R\},\\
\widetilde{w}^{sc}(x;z)-w^{re}(x;z)~~~~\mbox{in}~\Omega\cap\{|x|<R\}.
\end{array}
\begin{array}{l}
\end{array}\right.
\]
It is obvious that $w^{sc}$ is singular in $\Omega\cap\Omega'$ when $z\in\Omega\cap\Omega'$, where $\Omega':=\{x': x\in\Omega\}$. In addition, we still require the scattered field $w^{sc}(x;z)$ to satisfy the radiation condition
\be\label{3.3}
\lim\limits_{r\rightarrow\infty}r\left\{\partial_{r} w^{sc}(x;z)-ik w^{sc}(x;z)\right\}=0,
\en
where $r=|x|$, $x\in\Omega$, leading to the asymptotic behavior of $w^{sc}(x;z)$ as following:
\be\label{3.4}
w^{sc}(x;z)=\frac{e^{ik|x|}}{|x|}\left\{w^{sc}_{\infty}(\hat{x};z)+O\left(\frac{1}{|x|}\right)\right\},~~|x|\rightarrow\infty, x\in\Omega.
\en
Here $w_{\infty}^{sc}(\hat{x})$, $\hat{x}\in\mathbb{S}_{+}^{2}$ is the far-field pattern of $w^{sc}$.

In the following, we will introduce the mixed reciprocity relation between the scattered field of a plane wave and the far-field pattern excited by a point source wave. The mixed reciprocity relation was first proposed by Potthast \cite[Theorem 2.1.4]{Pot2001} for bounded obstacles. We refer to  Colton and Kress \cite[Theorem 5.6]{CK2013} for the application of the mixed reciprocity relation to a simplified uniqueness proof with infinitely many plane waves. In this paper, we consider acoustic wave scattering from unbounded surfaces, which is different from the arguments used in \cite{CK2013, Pot2001} for bounded obstacles. Following the lines of \cite[Theorem 3.16]{CK2013}, below we present the mixed reciprocity relation for scattering from locally perturbed rough surfaces, the proof of which mainly relies on Green's identities.
\begin{lem}\label{mixed}
We have the mixed reciprocity relation
$$4\pi w^{sc}_{\infty}(-d;z)=u^{sc}(z;d),$$
$where~~d\in \mathbb{S}^{2}_{-},~z\in \Omega.$
\end{lem}
\begin{rem}\label{remark 2.2}
\begin{itemize}
\item[(i)]
Difficulties in proving Lemma \ref{mixed} arise from the fact that the reflected and scattered fields corresponding to a point source incidence are singular if the source position $z$ is located in $\Omega\cap\Omega'$ (that is, $z'\in \Omega$), which will be particularly treated in the proof below. However, one can also avoid this by firstly proving Lemma \ref{mixed} for $z\in\Omega$ such that $z'\notin \Omega$ and then applying the analyticity of $w_\infty^{sc}(\cdot, z)$ and $u^{sc}(z; \cdot)$ in $z\in \Omega$.
\item[(ii)]
The reciprocity relation shown in Lemma \ref{mixed} together with the fact that
\ben
4\pi w^{re}_\infty(-d;z)=u^{re}(z;d)=-e^{ik d\cdot z'}
\enn yields the relation $4\pi \widetilde{w}^{sc}_{\infty}(-d;z)=\widetilde{u}^{sc}(z;d)$, where $\widetilde{u}^{sc}=u^{re}+u^{sc}$.
\end{itemize}
\end{rem}

\begin{proof}
Let $\Lambda_{R}=\Gamma\cap B_R$ be the local perturbation and let $\nu$ be the unit normal vector at $\Gamma$ directed into upper half space. Obviously,
\be\label{3.5}
\widetilde{u}^{in}(x;d)=\widetilde{w}^{in}(x;z)=0~~~~\mbox{on} ~\{x_{3}=0\},
\en
where $\widetilde{w}^{in}(x;z):=w^{in}(x;z)+w^{re}(x;z)$, $z\in \Omega$. It follows from the Dirichlet boundary condition and \eqref{3.5} that
\be\label{3.6}
u^{sc}(x;d)=w^{sc}(x;z)=0~~~~\mbox{on}~\Gamma\cap \{|x|>R\}.
\en

Since $z\in \Omega$, our proof can be divided into three cases: $(\romannumeral1)$ $z\in \Omega\backslash D_{R}$, where $D_{R}:=\{x\in \Omega:|x|<R\}$; $(\romannumeral2)$ $z\in D_{R}\backslash\Omega'$; $(\romannumeral3)$ $z\in \Omega\cap \Omega'$. Since the reflected field is singular in the case (iii), below we shall discuss the last case only.

 Choose $R_{1}>R$ and set $D_{R_{1}}:=\{x\in \Omega: |x|<R_{1}\}$. Applying the Green's second theorem of $u^{sc}$ and $w^{sc}$ in $D_{R_{1}}\backslash \overline{D_{R}}$ and using \eqref{3.6}, we obtain
$$0=\int_{D_{R_{1}}\backslash \overline{D_{R}}}\{u^{sc}(x;d)\Delta w^{sc}(x;z)-w^{sc}(x;z)\Delta u^{sc}(x;d)\}dx$$
$$~~~~~~~~~~~~=(\int_{S_{R_{1}}^{+}}-\int_{S_{R}^{+}})\{u^{sc}(x;d)\frac{\partial w^{sc}(x;z)}{\partial \nu(x)}-w^{sc}(x;z)\frac{\partial u^{sc}(x;d)}{\partial \nu(x)}\}ds(x),$$
where $S_{R_1}^{+}:=\{|x|=R_1:x_3>0\}$.
Letting $R_{1}\rightarrow \infty$ in the previous identity and using the radiation condition of $u^{sc}$ and $w^{sc}$, we get
$$0=\int_{S_{R}^{+}}\{u^{sc}(x;d)\frac{\partial w^{sc}(x;z)}{\partial \nu(x)}-w^{sc}(x;z)\frac{\partial u^{sc}(x;d)}{\partial \nu(x)}\}ds(x).$$
Since $w^{sc}(x;z)$ is singular at $x=z'\in \Omega\cap \Omega'$,  applying the Green's formula of $u^{sc}$ and $w^{sc}$ to $D_{R}$ yields
\be \nonumber
&&\int_{\Lambda_{R}}\{u^{sc}(x;d)\frac{\partial w^{sc}(x;z)}{\partial \nu}-w^{sc}(x;z)\frac{\partial u^{sc}(x;d)}{\partial \nu}\}ds(x)\\ \label{3.7}
&&=\int_{S_{R}^{+}}\{u^{sc}(x;d)\frac{\partial w^{sc}(x;z)}{\partial \nu}-w^{sc}(x;z)\frac{\partial u^{sc}(x;d)}{\partial \nu}\}ds(x)+u^{sc}(z';d)\\ \nonumber
&&=u^{sc}(z';d).
\en

Using the Green's second theorem of $u^{sc}$ and $\widetilde{w}^{in}$ yields
$$0=(\int_{S_{R_{1}}^{+}}-\int_{S_{R}^{+}})\{u^{sc}(x;d)\frac{\partial \widetilde{w}^{in}(x;z)}{\partial \nu(x)}-\widetilde{w}^{in}(x;z)\frac{\partial u^{sc}(x;d)}{\partial \nu(x)}\}ds(x).$$
As the proof of \eqref{3.7}, we obtain
\be\label{3.8}
0=\int_{S_{R}^{+}}\{u^{sc}(x;d)\frac{\partial \widetilde{w}^{in}(x;z)}{\partial \nu(x)}-\widetilde{w}^{in}(x;z)\frac{\partial u^{sc}(x;d)}{\partial \nu(x)}\}ds(x).
\en
Because $z'\in \Omega$, it is obvious that $\widetilde{w}^{in}(x;z)$ is singular at $x=z$ and $x=z'$. Similarly, using the Green's formula of $u^{sc}$ and $\widetilde{w}^{in}$ in $D_{R}$, we get
\be\label{3.9}
u^{sc}(z;d)-u^{sc}(z';d)=\int_{\Lambda_{R}}\left\{u^{sc}(x;d)\frac{\partial \widetilde{w}^{in}(x;z)}{\partial \nu(x)}-\widetilde{w}^{in}(x;z)\frac{\partial u^{sc}(x;d)}{\partial \nu(x)}\right\}ds(x).
\en

Let $C_{R}$ be the domain enclosed by $\Lambda_{R}$ and $\{x_{3}=0\}$. Then $C_{R}$ may consist of several connected components and $\partial C_{R}\subset \Lambda_{R}\cup\{x_{3}=0\}$. By the definition of $C_{R}$ and \eqref{3.5}, we obtain
\be\nonumber
~~\int_{\Lambda_{R}}\left\{\widetilde{u}^{in}(x;d)\frac{\partial \widetilde{w}^{in}(x;z)}{\partial \nu(x)}-\widetilde{w}^{in}(x;z)\frac{\partial \widetilde{u}^{in}(x;d)}{\partial \nu(x)}\right\}ds(x)\\ \label{3.10}
=-\int_{\partial C_{R}}\left\{\widetilde{u}^{in}(x;d)\frac{\partial \widetilde{w} ^{in}(x;z)}{\partial \nu(x)}-\widetilde{w}^{in}(x;z)\frac{\partial \widetilde{u}^{in}(x;d)}{\partial \nu(x)}\right\}ds(x).
\en
Applying Green's formula of $\widetilde{u}^{in}$ and $\widetilde{w}^{in}$ to $C_{R}$, we obtain from (\ref{3.10}) that
\be\label{3.11}
\widetilde{u}^{in}(z;d)=\int_{\Lambda_{R}}\left\{\widetilde{u}^{in}(x;d)\frac{\partial \widetilde{w}^{in}(x;z)}{\partial \nu(x)}-\widetilde{w}^{in}(x;z)\frac{\partial \widetilde{u}^{in}(x;d)}{\partial \nu(x)}\right\}ds(x).
\en

Suppose $|y|>R$, $y\in\Omega$. As the proof of \eqref{3.7}, we get
\be\label{3.12}
w^{sc}(y;z)=\int_{S_{R}^{+}}\left\{w^{sc}(x;z)\frac{\partial \widetilde{w}^{in}(x;y)}{\partial \nu(x)}-\widetilde{w}^{in}(x;y)\frac{\partial w^{sc}(x;z)}{\partial \nu(x)}\right\}ds(x).
\en
Similar to \eqref{3.9}, it follows from \eqref{3.12} that
\be\label{3.13}
w^{sc}(y;z)+\widetilde{w}^{in}(z;y)=\int_{\Lambda_{R}}\left\{w^{sc}(x;z)\frac{\partial \widetilde{w}^{in}(x;y)}{\partial \nu(x)}-\widetilde{w}^{in}(x;y)\frac{\partial w^{sc}(x;z)}{\partial \nu(x)}\right\}ds(x).
\en
By the definition of $\widetilde{w}^{in}(x;y)$, there has the asymptotic behavior
\be\label{3.14}
\widetilde{w}^{in}(x;y)=\frac{1}{4\pi}\frac{e^{ik|y|}}{|y|}(e^{-ik\hat{y}\cdot x}-e^{-ik\hat{y}\cdot x'})+O(\frac{1}{|y|}),~~|y|\rightarrow\infty.
\en
Substituting the asymptotic behavior of $w^{sc}$ and \eqref{3.14} into \eqref{3.13} and taking $|y|\rightarrow\infty$, we get
$$4\pi w^{sc}_{\infty}(\hat{y};z)+\widetilde{u}^{in}(z;-\hat{y})=\int_{\Lambda_{R}}\left\{w^{sc}(x;z)\frac{\partial \widetilde{u}^{in}(x;-\hat{y})}{\partial \nu(x)}-\widetilde{u}^{in}(x;-\hat{y})\frac{\partial w^{sc}(x;z)}{\partial \nu(x)}\right\}ds(x).$$
Taking $\hat{y}=-d$, then the previous formula becomes
\be\label{3.15}
4\pi w^{sc}_{\infty}(-d;z)+\widetilde{u}^{in}(z;d)=\int_{\Lambda_{R}}\left\{w^{sc}(x;z)\frac{\partial \widetilde{u}^{in}(x;d)}{\partial \nu(x)}-\widetilde{u}^{in}(x;d)\frac{\partial w^{sc}(x;z)}{\partial \nu(x)}\right\}ds(x).
\en

On the other hand, adding \eqref{3.7} and \eqref{3.9} gives
\be\label{3.16}
u^{sc}(z;d)=\int_{\Lambda_{R}} u^{sc}(x;d)\frac{\partial w(x;z)}{\partial \nu(x)}-w(x;z)\frac{\partial u^{sc}(x;d)}{\partial \nu(x)}\}ds(x).
\en
Substracting \eqref{3.11} from \eqref{3.15}, we obtain
\be\label{3.17}
4\pi w^{sc}_{\infty}(-d;z)=\int_{\Lambda_{R}}\left\{w(x;z)\frac{\partial \widetilde{u}^{in}(x;d)}{\partial \nu}-\widetilde{u}^{in}(x;d)\frac{\partial w(x;z)}{\partial \nu}\right\}ds(x).
\en
Finally, substracting \eqref{3.16} from \eqref{3.17} and using $w=u=0~~\mbox{on}~\Gamma$, we have
$$4\pi w^{sc}_{\infty}(-d;z)-u^{sc}(z;d)=0,$$
which completes the proof Lemma \ref{mixed} in case $(\romannumeral3)$.

The other two cases can be treated analogously.

\end{proof}

Next, we state the symmetry of the scattered field for point source waves, that is, the scattered field of a point source wave remains unchanged when the incident direction and the observed direction are exchanged. Lemma \ref{sym} below extends the result of \cite[Theorem 3.17]{CK2013} from bounded obstacles to unbounded obstacles in a half space with local perturbations.
\begin{lem}\label{sym}
For scattering from locally perturbed sound-soft surfaces of a point source wave, we have the symmetry relation
$$w^{sc}(x;y)=w^{sc}(y;x),~~x,y\in \Omega.$$
\end{lem}
\begin{proof}
Since $x,y\in\Omega$, our proof can be divided into the following nine cases:\\
$(\romannumeral1)$ $x\in \Omega\backslash D_{R}$, $y\in \Omega\backslash D_{R}$;~~~~~~~~~~$(\romannumeral2)$ $x\in\Omega\backslash D_{R}$, $y\in D_{R}\backslash \Omega'$;\\
$(\romannumeral3)$ $x\in \Omega\backslash D_{R}$, $y\in \Omega \cap \Omega'$;~~~~~~~~$(\romannumeral4)$ $x\in D_{R}\backslash \Omega'$, $y\in \Omega\backslash D_{R}$;\\
$(\romannumeral5)$ $x\in D_{R}\backslash \Omega'$, $y\in D_{R}\backslash \Omega'$;~~~~~~~~$(\romannumeral6)$ $x\in D_{R}\backslash \Omega'$, $y\in \Omega \cap \Omega'$;\\
$(\romannumeral7)$ $x\in \Omega \cap \Omega'$, $y\in \Omega\backslash D_{R}$;~~~~~~~$(\romannumeral8)$ $x\in \Omega \cap \Omega'$, $y\in D_{R}\backslash \Omega'$;\\
$(\romannumeral9)$ $x\in \Omega \cap \Omega'$, $y\in \Omega \cap \Omega'$.

We only consider case $(\romannumeral9)$, while the remaining eight cases can be verified in the same manner. Note again that in the case $(\romannumeral9)$ the reflected and scattered fields are both singular in $\Omega \cap \Omega'$, but their sum turns out to be non-singular.

First of all, we choose $R_{1}>R$ such that $R_{1}>\max\{|x|,|y|\}$. Since $y\in \Omega \cap \Omega'$, the function $w^{sc}(x;y)$ is singular at $x=y'$. Applying the Green's second theorem of $w^{sc}(\cdot;y)$ and $\widetilde{w}^{in}(\cdot;x)$ to $D_{R_{1}}\backslash \overline{D_{R}}$, we get
$$0=(\int_{S_{R_{1}}^{+}}-\int_{S_{R}^{+}})\{w^{sc}(\cdot;y)\frac{\partial \widetilde{w}^{in}(\cdot;x)}{\partial \nu}-\widetilde{w}^{in}(\cdot;x)\frac{\partial w^{sc}(\cdot;y)}{\partial \nu}\}ds.$$
Letting $R_{1}\rightarrow\infty$ and making use of the radiation conditions of $w^{sc}(\cdot;y)$ and $\widetilde{w}^{in}(\cdot;x)$, we obtain
\be\label{3.18}
0=\int_{S_{R}^{+}}\{w^{sc}(\cdot;y)\frac{\partial \widetilde{w}^{in}(\cdot;x)}{\partial \nu}-\widetilde{w}^{in}(\cdot;x)\frac{\partial w^{sc}(\cdot;y)}{\partial \nu}\}ds.
\en
It is obvious that $w^{sc}(\cdot;y)$ and $\widetilde{w}^{in}(\cdot;x)$ are singular in $D_{R}$. As in the proof of \eqref{3.8}, it follows from \eqref{3.18} that
\be\label{3.19}
\widetilde{w}^{sc}(x;y)-\widetilde{w}^{sc}(x';y)=\int_{\Lambda_{R}}\{w^{sc}(\cdot;y)\frac{\partial \widetilde{w}^{in}(\cdot;x)}{\partial \nu}-\widetilde{w}^{in}(\cdot;x)\frac{\partial w^{sc}(\cdot;y)}{\partial \nu}\}ds.
\en
In $D_{R_{1}}\backslash \overline{D_{R}}$, similar to the proof of \eqref{3.18}, we obtain
\be\label{3.20}
0=\int_{S_{R}^{+}}\{\widetilde{w}^{in}(\cdot;y)\frac{\partial \widetilde{w}^{in}(\cdot;x)}{\partial \nu}-\widetilde{w}^{in}(\cdot;x)\frac{\partial \widetilde{w}^{in}(\cdot;y)}{\partial \nu}\}ds.
\en
In addition, $\widetilde{w}^{in}(\cdot;y)$ is singular in $D_{R}$ due to $y\in \Omega \cap \Omega'$. Using Green's formula of $\widetilde{\omega}^{in}(\cdot,x)$ and $\widetilde{\omega}^{in}(\cdot,y)$ in $D_{R}$ and combining with \eqref{3.20}, we get
\be\label{3.21}
0=\int_{\Lambda_{R}}\{\widetilde{w}^{in}(\cdot;y)\frac{\partial \widetilde{w}^{in}(\cdot;x)}{\partial \nu}-\widetilde{w}^{in}(\cdot;x)\frac{\partial \widetilde{w}^{in}(\cdot;y)}{\partial \nu}\}ds.
\en
Arguing the same as in \eqref{3.21}, we obtain
\be\label{3.22}
\widetilde{w}^{sc}(x';y)-\widetilde{w}^{sc}(y';x)=\int_{\Lambda_{R}}\{w^{sc}(\cdot;y)\frac{\partial w^{sc}(\cdot;x)}{\partial \nu}-w^{sc}(\cdot;x)\frac{\partial w^{sc}(\cdot;y)}{\partial \nu}\}ds.
\en
Analogously, applying Green's formula of $w^{sc}(\cdot;x)$ and $\widetilde{w}^{in}(\cdot;y)$ in $D_{R_{1}}\backslash \overline{D_{R}}$, we have
\be\label{3.23}
\widetilde{w}^{sc}(y;x)-\widetilde{w}^{sc}(y';x)=\int_{\Lambda_{R}}\{w^{sc}(\cdot;x)\frac{\partial \widetilde{w}^{in}(\cdot;y)}{\partial \nu}-\widetilde{w}^{in}(\cdot;y)\frac{\partial w^{sc}(\cdot;x)}{\partial \nu}\}ds.
\en

Adding \eqref{3.19} and \eqref{3.21}, we get
\be\label{3.24}
\widetilde{w}^{sc}(x;y)-\widetilde{w}^{sc}(x';y)=\int_{\Lambda_{R}}\{w(\cdot;y)\frac{\partial \widetilde{w}^{in}(\cdot;x)}{\partial \nu}-\widetilde{w}^{in}(\cdot;x)\frac{\partial w(\cdot;y)}{\partial \nu}\}ds.
\en
Subtracting \eqref{3.22} from \eqref{3.23}, we obtain
\be\label{3.25}
\widetilde{w}^{sc}(y;x)-\widetilde{w}^{sc}(x';y)=\int_{\Lambda_{R}}\{w^{sc}(\cdot;x)\frac{\partial w(\cdot;y)}{\partial \nu}-w(\cdot;y)\frac{\partial w^{sc}(\cdot;x)}{\partial \nu}\}ds.
\en
Finally, subtracting \eqref{3.24} from \eqref{3.25} and using $w=u=0$ on $\Gamma$, we have $w^{sc}(y;x)=w^{sc}(x;y)$.

\end{proof}
\begin{rem}\label{remark 2.2}
Lemma \ref{sym} implies that $w(x;y)=w(y;x)$, because $\widetilde{w}^{in}(x;y)=\widetilde{w}^{in}(y;x)$.  This gives the symmetry of the Green's function to locally perturbed rough surface scattering problems.
\end{rem}
The symmetry relation shown in Lemma \ref{sym} will be used in our uniqueness proof of Theorem \ref{alldirection} below. It is seen from Theorem \ref{alldirection} that an unbounded scattering interface with local perturbations can be uniquely determined by an infinite number of incident waves with all directions and a fixed wave number. 
We shall carry out the proof following the arguments of \cite{KK1993} but generalizes it to be applicable for  scattering problems in a half space.
\begin{thm}\label{alldirection}
Assume $\Gamma_{1}$ and $\Gamma_{2}$ are two sound-soft surfaces with local perturbations such that their far-field patterns $u^{sc}_{1,\infty}(\hat{x};d)$ and $u^{sc}_{2,\infty}(\hat{x};d)$ coincide for all directions $d\in \mathbb{S}^{2}_{-}$ and one fixed wave number. Then $\Gamma_{1}=\Gamma_{2}.$
\end{thm}
\begin{proof}
We assume that $\Gamma_{j}$ $(j=1,2)$ is a surface with the local perturbation $\Lambda_{R}^{(j)}$. Let $u^{sc}_{j,\infty}(\hat{x};d)$ be the far-field pattern of the scattered field $u_{j}^{sc}(x;d)$, and let $w^{sc}_{j}(x;z)$ be the scattered field corresponding to the incoming point source, respectively. Supposing that
$$u^{sc}_{1,\infty}(\hat{x};d)=u^{sc}_{2,\infty}(\hat{x};d), ~~\mbox{for all}~\hat{x}\in \mathbb{S}^{2}_{+},~d\in \mathbb{S}^{2}_{-},$$
we need to prove that $\Gamma_{1}=\Gamma_{2}$.

By Theorem \ref{farscat}, we have $u^{sc}_{1}(x;d)=u^{sc}_{2}(x;d)$ for all $x\in G:=\Omega_{1}\cap \Omega_{2}$, $d\in \mathbb{S}^{2}_{-}$. Applying the mixed reciprocity relation (see Lemma \ref{mixed}), we obtain
$$w^{sc}_{1,\infty}(\hat{y};x)=w^{sc}_{2,\infty}(\hat{y};x),$$
for all $\hat{y}\in \mathbb{S}^{2}_{+},~x\in G$. Similar to the proof of Theorem \ref{farscat}, for point source wave, we can get
$$w^{sc}_{1}(y;x)=w^{sc}_{2}(y;x)~~\mbox{for all}~ x,y\in G.$$

It is sufficient to prove that $\Lambda_{R}^{(1)}=\Lambda_{R}^{(2)}$. Assume on the contrary that $\Lambda_{R}^{(1)}\neq \Lambda_{R}^{(2)}$. We can always find a point $x^{\ast} \in \partial G$ such that $x^{\ast}\in \Lambda_{R}^{(1)}$ but $x^{\ast} \notin \Lambda_{R}^{(2)}$ (see Figure \ref{fig:2} below).
\begin{figure}[ht]
\centering
\includegraphics[scale=0.6]{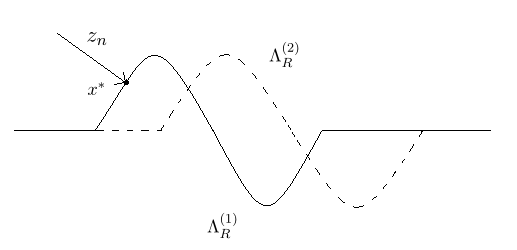}
\caption{The solid line is $\Lambda_{R}^{(1)}$, the dotted line is $\Lambda_{R}^{(2)}$.}
\label{fig:2}
\end{figure}
Define
$$z_{n}:=x^{\ast}+\frac{1}{n} \nu (x^{\ast}),  ~n=1,2,\cdots$$
which lie in $G$ for sufficiently large $n$. On the one hand, since $x^{*}\in \Omega_{2}$, we know that $w^{sc}_{2}(\cdot;x^{\ast})$ is continuously differentiable in a neighborhood of $x^{\ast} \notin \Gamma_{2}$ from the Green's formula of $w^{sc}_{2}$ in $\Omega_{2}$. Owing to the symmetry relation of $w_{2}^{sc}(x^{\ast};\cdot)$ and the well-posedness of the direct scattering problem with the Dirichlet boundary condition on $\Gamma_{2}$, we have
$$\lim\limits_{n\rightarrow\infty}w^{sc}_{2}(x^{\ast}; z_{n})=\lim\limits_{n\rightarrow\infty}w^{sc}_{2}(z_{n};x^{\ast})=w^{sc}_{2}(x^{\ast}; x^{\ast})<\infty.$$
On the other hand, using the Dirichlet boundary condition on $\Gamma_{1}$, we get
$$\lim\limits_{n\rightarrow\infty}w^{sc}_{1}(x^{\ast};z_{n})=-\lim\limits_{n\rightarrow\infty}\widetilde{w}^{in}_{1}(x^{\ast};z_{n})=\infty.$$
This contradicts the relation
$$w^{sc}_{1}(x^{\ast}; z_{n})=w^{sc}_{2}(x^{\ast}; z_{n})~~\mbox{for all sufficiently large}~n.$$
Therefore, we get $\Lambda_{R}^{(1)}=\Lambda_{R}^{(2)}$.

\end{proof}
\subsubsection{Uniqueness with a single incident wave}\label{sec:unising}
In this subsection, we keep the symbols used in the previous section and suppose that $f$ is a piecewise linear function. In this case, $\Lambda_{R}$ is called a local perturbation of polyhedral type, as shown in Figure \ref{fig:3}. We shall prove in Theorems \ref{polyplane} and \ref{polypoint} that a polyhedral surface can be uniquely determined by an incident plane wave or a single point source wave.
\begin{figure}[ht]
\centering
\includegraphics[scale=0.4]{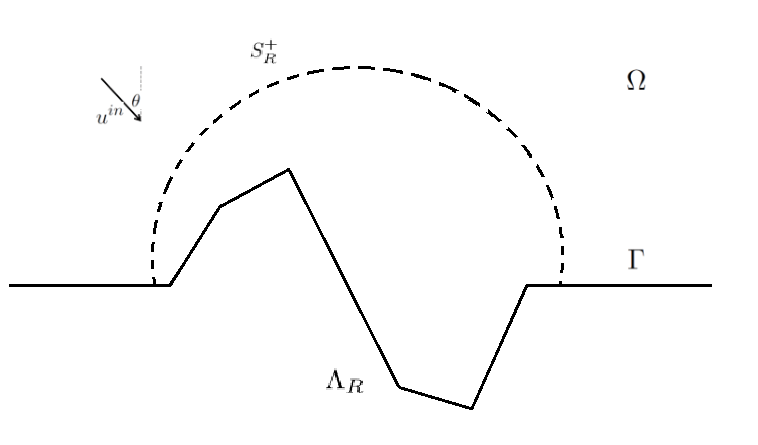}
\caption{Illustration of a locally perturbed polyhedral surface.}
\label{fig:3}
\end{figure}
\begin{thm}\label{polyplane}
Under the Dirichlet boundary condition, a polyhedral surface can be uniquely determined by the far-field pattern of an incident plane wave.
\end{thm}
\begin{proof}
Denoting $\Gamma_{1}$ and $\Gamma_{2}$ two surfaces with local perturbations of polyhedral type (see Figure \ref{fig:4}).
\begin{figure}[ht]
\centering
\includegraphics[scale=0.6]{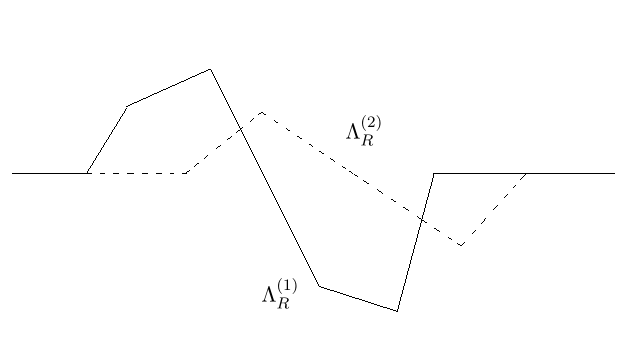}
\caption{The solid line is $\Lambda_{R}^{(1)}$, the dotted line is $\Lambda_{R}^{(2)}$.}
\label{fig:4}
\end{figure}
By our assumptions, the far-field patterns of $\Gamma_{1}$ and $\Gamma_{2}$ are identical, i.e.,
$$u_{1,\infty}^{sc}(\hat{x};d)=u_{2,\infty}^{sc}(\hat{x};d),~~\hat{x}\in \mathbb{S}^{2}_{+},~d\in \mathbb{S}^{2}_{-}.$$
It follows from Theorem \ref{farscat} that
\be\label{3.26}
u^{sc}_{1}(x;d)=u^{sc}_{2}(x;d),~~x\in \Omega_{1}\cap \Omega_{2},~d\in \mathbb{S}^{2}_{-},
\en
implying that
\be\label{3.27}
u_{1}(x;d)=u_{2}(x;d),~~x\in \Omega_{1}\cap\Omega_{2},~d\in \mathbb{S}^{2}_{-}.
\en

Now we assume that $\Gamma_{1}\neq\Gamma_{2}$. Without loss of generality, we assume that $S:=(\Gamma_{1}\backslash\Gamma_{2})\cap \partial(\Omega_{1}\cap \Omega_{2})\neq \emptyset$. By \eqref{3.27}, we get
$$u_{1}(x;d)=u_{2}(x;d)~~\mbox{on}~S.$$
Recalling that $u_{1}(x;d)=0$ on $S$, we obtain
\be\label{3.28}
u_{2}(x;d)=0~~\mbox{on}~S.
\en

First, we define the nodal set of $u_2$. A nodal set $\Sigma$ consists of flat surfaces in $\Omega_{2}$ on which $u_{2}$ vanishes. Each flat surface of $\Sigma$ is called a cell. Obviously, a cell may be bounded or unbounded. By \eqref{3.28}, there exists a bounded cell $\Pi\subset S$. Denote by $\widetilde{\Pi}$ the maximum extension of $\Pi$ in $\Omega_{2}$. Since $u_{2}$ is analytic in $\Omega_{2}$, $\widetilde{\Pi}$ also belongs to $\Sigma$. Without loss of generality, we may assume that $\widetilde{\Pi}$ is unbounded in the positive $x_{3}$ direction. In fact, if otherwise, one can always apply the reflection and path arguments to find a new cell which is unbounded in the positive $x_{3}$ direction. We note that this step has been extensively studied in the literature for both bounded and unbounded obstacles (see e.g.,\cite{AR2005, EY2006, Hu2012, HL2014, LZ2006}) and we thus omit the details for simplicity. Readers can be referred to Section \ref{elec} for the path and reflection arguments in electromagnetic scattering.

In view of the radiation condition of $u^{sc}_{2}$, we have
\be\label{3.29}
\lim\limits_{|x| \rightarrow\infty}|u_{2}^{sc}(x;d)|=0,~~x\in\widetilde{\Pi}.
\en
Without loss of generality, we may suppose that
\be\label{3.30}
\widetilde{\Pi}:=\{x_{3}=ax_{1}+bx_{2}+c,~ a,b,c\in\mathbb{R}, a^{2}+b^{2}\neq0\}\cap\{x_{3}\geq L\},
\en
where $L\in\mathbb{R}$ is some positive number. In fact, if $a^{2}+b^{2}=0$ in \eqref{3.30}, one can reflect $\Gamma_{2}$ with respect to the flat surface $\{x_{3}=c\}$. Consequently, by reflection principle for the Helmholtz equation, one can find another unbounded flat surface of the form $\widetilde{\Pi}$.

For $x\in\widetilde{\Pi}$, it holds that
\be\label{3.31}
u^{in}(x;d)+u^{re}(x;d)=e^{ik(\alpha x_{1}+\beta x_{2})}[e^{-ik\gamma x_3}-e^{ik\gamma x_3}],
\en
where $\gamma\in(0,1]$. In view of the definition of $\widetilde{\Pi}$ and \eqref{3.31}, one can find that
\be\label{3.32}
\lim\limits_{|x|\rightarrow\infty} |u_{2}^{sc}(x;d)|=\lim\limits_{|x|\rightarrow \infty}|-u^{in}_{2}(x;d)-u^{re}_{2}(x;d)|\neq0,~~x\in\widetilde{\Pi}.
\en
This is a contradiction to \eqref{3.29}. Therefore $\Gamma_{1}=\Gamma_{2}$.

\end{proof}
\begin{thm}\label{polypoint}
Under the Dirichlet boundary condition, a polyhedral surface can be uniquely determined by the far-field pattern of one point source wave.
\end{thm}
\begin{proof}
We assume that $\Gamma_{1}$ and $\Gamma_{2}$ are two surfaces with local perturbations of polyhedral type. Let $y\in\Omega_{1}\cap\Omega_{2}$ be fixed and suppose that $\Gamma_{1}$ and $\Gamma_{2}$ have the same far field pattern, i.e.,
$$w_{1,\infty}^{sc}(\hat{x};y)=w_{2,\infty}^{sc}(\hat{x};y),~~\hat{x}\in \mathbb{S}^{2}_{+},~y\in\Omega_{1}\cap\Omega_{2}.$$
Recalling Theorem \ref{farscat}, we get
\be\label{3.33}
w_{1}^{sc}(x;y)=w_{2}^{sc}(x;y),~w_{1}(x;y)=w_{2}(x;y),~x\in\Omega_{1}\cap\Omega_{2}.
\en

Assuming $\Gamma_{1}\neq\Gamma_{2}$, we shall prove Theorem \ref{polypoint} by deriving a contradiction. We assume that $S:=(\Gamma_{1}\backslash\Gamma_{2})\cap\partial(\Omega_{1}\cap\Omega_{2})\neq\emptyset$. Due to the regularity and \eqref{3.33}, we obtain
$$w_{1}(x;y)=w_{2}(x;y)~~\mbox{on}~S.$$
Combining with the boundary conditions $w_{1}(x;y)=0$ on $S$, we have
\be\label{3.34}
w_{2}(x;y)=0~~\mbox{on}~S.
\en
Here, the nodal set is defined by flat surfaces (cells) in $\Omega_{2}$ on which $w_{2}$ vanishes. Our purpose is to find a bounded cell $\Pi$ such that its maximum extension $\widetilde{\Pi}$ in $\Omega_{2}$ fulfills the following condition: the reflection of the point source $y$ with respect to
 $\widetilde{\Pi}$ still lies in $\Omega_2$.
In fact, by \eqref{3.34} there exists a bounded cell $\Pi_{1}\subset S$. Let $\widetilde{\Pi}_{1}$ be the maximum extension of $\Pi_{1}$ in $\Omega_{2}$. Since $w_{2}(x;y)$ is real analytic in $\Omega_{2}\backslash\{y\}$, it follows that
\be\label{3.35}
w_{2}(x;y)=0~~\mbox{on}~\widetilde{\Pi}_{1}.
\en
Repeating the reflection and path arguments used in \cite{Hu2012, HL2014}, one can always find a desired flat surface whose maximum extension satisfies the above condition.

Since $y^*$ is the symmetric point of $y$ with respect to the cell $\widetilde{\Pi}$, in view of the reflection principle for the Helmholtz equation, we get
$$w_{2}(x;y)=-w_{2}(x;y^{*}),~~x\in\Omega_{2}\backslash \{y\}.$$
It is obvious that
$$\lim\limits_{x\rightarrow y^{*}}|w_{2}(x;y)|=\lim\limits_{x\rightarrow y^{*}}|-w_{2}(x;y^{*})|=\lim\limits_{x\rightarrow y^{*}}|\widetilde{w}^{in}(x;y^*)+w^{sc}(x;y^*)|=\infty,$$
due to the singularity of $\widetilde{w}^{in}(x;y^*)$ at $x=y^*$.
On the other hand, it holds that
$$\lim\limits_{x\rightarrow y^{*}}|w_{2}(x;y)|=|w_{2}(y^{*};y)|=|\widetilde{w}^{in}(y^*;y)+w^{sc}(y^*;y)|<\infty.$$
This contradiction implies $\Gamma_{1}=\Gamma_{2}$.

\end{proof}
\subsection{Uniqueness in determining sound-hard surfaces}\label{sec:hard}
The aim of this subsection is to carry over the uniqueness results established in Section \ref{sec:soft} to sound-hard surfaces. We sketch the proofs, since most arguments are similar to the sound-soft case.
We first prove the Rellich lemma in a half space under the Neumann boundary condition.
\begin{thm}\label{hardplane}
Assume $\Gamma_{1}$ and $\Gamma_{2}$ are two sound-hard surfaces with local perturbations such that their far-field patterns coincide for the plane wave with the direction $d\in \mathbb{S}^{2}_{-}$. Then
$$u^{sc}_{1}(x;d)=u^{sc}_{2}(x;d)~~for~all~x\in\Omega_{1}\cap \Omega_{2},~d\in\mathbb{S}^{2}_{-}.$$
\end{thm}
\begin{proof}
Apparently, $u^{sc}_{1}(x)$ satisfies the Helmholtz equation in $\Omega_{1}$. Applying even extension of $u^{sc}_{1}(x)$ from $\mathbb{R}^{3}_{+}\backslash \overline{B^{+}_{R}}$ to $\mathbb{R}^{3}\cap\{|x|>R\}$, we thus obtain
\be\label{4.1}
 W^{sc}_{1}(x_{1},x_{2},x_{3})=
\left\{\begin{array}{l}
u^{sc}_{1}(x_{1},x_{2},x_{3}),~~|x|>R, ~x_{3}>0,\\
\partial _{x_{3}}u^{sc}_{1}(x_{1},x_{2},0),~~|x|>R,~x_{3}=0,\\
u^{sc}_{1}(x_{1},x_{2},-x_{3}),~~|x|>R,~x_{3}<0.
\end{array}
\begin{array}{l}
\end{array}\right.
\en
Arguing the same as in the proof of Theorem \ref{farscat}, we get
$$u_{1}^{sc}(x;d)=u_{2}^{sc}(x;d)~~\mbox{in}~\Omega_{1}\cap \Omega_{2}$$
for any $d\in \mathbb{S}^{2}_{-}$.
\end{proof}
\begin{thm}\label{harduni}
Assume $\Gamma_{1}$ and $\Gamma_{2}$ are two sound-hard surfaces with local perturbations such that their far-field patterns $u^{sc}_{1,\infty}(\hat{x};d)$ and $u^{sc}_{2,\infty}(\hat{x};d)$ coincide for infinite incident plane waves with all directions $d\in \mathbb{S}^{2}_{-}$ and one fixed wave number. Then $\Gamma_{1}=\Gamma_{2}.$
\end{thm}
\begin{proof}
Recalling that the boundary condition $\partial_{\nu}u_{j}(x;d)=\partial_{\nu}w_{j}(x;z)=0~\mbox{on}~\Gamma_{j} (j=1,2)$, we can use the same method as in Lemma \ref{mixed} to get the mixed reciprocity relation
\be\label{4.2}
4\pi w^{sc}_{j,\infty}(-d;z)=u^{sc}_{j}(z;d),
\en
where $z\in\Omega$, $d\in\mathbb{S}^{2}_{-}$. In view of the boundary condition $\partial_{\nu}w_{j}(\cdot,x)=\partial_{\nu}w_{j}(\cdot,y)=0~\mbox{on}~\Gamma_{j}$, we can get the symmetry relation
\be\label{4.3}
w^{sc}_{j}(x;y)=w^{sc}_{j}(y;x),~~x, y \in \Omega,
\en
following the lines in the proof of Lemma \ref{sym}. Consequently, we get $w^{sc}_{1}(y;x)=w^{sc}_{2}(y;x)$ for all $y, x \in \Omega$. The proof can be carried out similarly to that of Theorem \ref{alldirection} by the idea of blow-up of point sources to reach $\Gamma_{1}=\Gamma_{2}$.

\end{proof}
\begin{thm}\label{polyuni}
Under the Neumann boundary condition, a polyhedral surface can be uniquely determined by the far-field pattern of a point source wave or a plane wave with an arbitrary incident direction.
\end{thm}
\begin{proof}
Supposing on the contrary that $\Gamma_{1}\neq\Gamma_{2}$, we shall verify the uniqueness results for  point source and plane waves separately.

\emph{Case 1. A polyhedral surface can be uniquely determined by a point source wave.}  In the Neumann case, the nodal set is defined by the flat surfaces (cells) with vanishing normal derivative of the total field. By reflection and path arguments, one can always find a cell $\Pi$ such that its extension $\widetilde{\Pi}$ fulfills the same condition in the proof of Theorem \ref{polypoint}, which gives arise to the conclusion.

\emph{Case 2. A polyhedral surface can be uniquely determined by a plane wave with an arbitrary incident direction.} Let the flat surface $\Pi$ and its extension $\widetilde{\Pi}$ be the same as those defined in the proof of Theorem \ref{polyplane}. In particular, $\widetilde{\Pi}$ can be extended to infinity in the positive $x_3$-direction and  $\partial_{\nu} u_{2}(x;d)=0~\mbox{on}~\widetilde{\Pi}$, where $\nu:=(\nu_{1},\nu_{2},\nu_{3})$ is the unit normal vector of $\widetilde{\Pi}$ directed into $\mathbb{R}^{3}_{+}$. Without loss of generality, we assume
\be\label{4.4}
\widetilde{\Pi}:=\{x_{3}=Ax_{1}+Bx_{2}+C:=F(\tilde{x}), A,B,C\in \mathbb{R}, A^{2}+B^{2}\neq 0\}\cap\{x_3\geq E\},
\en
where $E\in \mathbb{R}$ is some positive number. Obviously, $\nu=(-A\setminus h,-B\setminus h, 1\setminus h)$ is the unit normal vector of $\widetilde{\Pi}$, where $h:=\sqrt{A^2+B^2+1}$. 

In view of the expressions of the incident plane wave and its reflected wave, we get
\be\label{4.5}
\partial_{\nu}(u^{in}+u^{re})(x;d)=ike^{ik(\alpha x_{1}+\beta x_{2})}[e^{-ik\gamma x_{3}}(d\cdot \nu)+e^{ik\gamma x_{3}}(d'\cdot \nu)],
\en
where $d':=(\alpha,\beta,\gamma)$. From the construction of $\widetilde{\Pi}$, we know
$$0=\partial_{\nu}u_{2}(x;d)=ike^{ik(\alpha x_{1}+\beta x_{2})}[e^{-ik\gamma x_{3}}(d\cdot \nu)+e^{ik\gamma x_{3}}(d'\cdot \nu)]+\partial_{\nu} u^{sc}_{2}(x;d),~~x\in\widetilde{\Pi}.$$
It is obvious that $\lim\limits_{|x|\rightarrow\infty}\partial_{\nu}u_{2}^{sc}(x;d)=0,~x\in\widetilde{\Pi}$, due to the Sommerfeld radiation condition of $u^{sc}_{2}$. Hence,
\be\label{4.6}
0=\lim\limits_{|x|\rightarrow\infty}e^{ik(\alpha x_{1}+\beta x_{2})}[e^{-ik\gamma x_{3}}(d\cdot \nu)+e^{ik\gamma x_{3}}(d'\cdot \nu)],~x\in\widetilde{\Pi}.
\en
The relation \eqref{4.6} implies that
\be\label{4.7}
0=\lim\limits_{|\tilde{x}|\rightarrow\infty}[(d\cdot \nu)+e^{2ik\gamma F(\tilde{x})}(d'\cdot \nu)],~\tilde{x}\in\widetilde{\Pi}.
\en
Recalling that the incident angle $\varphi \in (-\frac{\pi}{2},\frac{\pi}{2})$ and $\theta \in (0, 2\pi)$, we can get $\alpha\in(-1,1)$, $\beta\in(-1,1)$ and $\gamma\in(0,1]$. Obviously, $d\cdot \nu$ and $d'\cdot \nu$ cannot be zero simultaneously.

Below, we will discuss two cases. Case $(\romannumeral1)$:~$d\cdot \nu=0$. In this case, $d'\cdot \nu\neq 0$, since $\nu_3=1/h\neq 0$.
Then, by \eqref{4.7} we find
\be\label{4.8}
0=\lim\limits_{|\tilde{x}|\rightarrow\infty}e^{2ik\gamma F(\tilde{x})}(d'\cdot \nu).
\en
We can always choose a sequence $x_{n}=(\widetilde{x}_{n},x_{3})\in\widetilde{\Pi}$ such that $\lim\limits_{|n|\rightarrow\infty}|F(\widetilde{x}_{n})|=+\infty$. For this sequence, it is easy to conclude that the limit on the right side of (4.8) does not exist, which leads to a contradiction.

Case $(\romannumeral2)$:~$d\cdot \nu\neq0$. From \eqref{4.7}, we get
\be\label{4.9}
-1=\lim\limits_{|\tilde{x}|\rightarrow\infty}e^{2ik\gamma F(\tilde{x})}\frac{d'\cdot \nu}{d\cdot \nu}.
\en
Similar to the proof of Case $(\romannumeral1)$, we can also derive a contradiction. This finishes the proof of $\Gamma_1=\Gamma_2$ for an incident plane wave.
\end{proof}

\section{Electromagnetic scattering problem}\label{elec}
This section is concerned with the inverse time-harmonic electromagnetic scattering problems of
determining perfect conductors from far-field measurement of the scattered
electric fields. The forward electromagnetic scattering from a locally perturbed rough surface is investigated in \cite{LWZ}.
We note that the uniqueness results with infinitely many plane waves (see Theorems \ref{alldirection} and \ref{harduni} in the acoustic case) carry over to  electromagnetic scattering problems straightforwardly.
Readers can be referred to \cite{K2002} and \cite[Chapter 7.1, Theorem 6.31]{CK2013} for more details in the case of a bounded conductor as well as the mixed reciprocity relation in electromagnetic scattering.
 Hence, in this section we shall pay our attention to the unique determination of polyhedral conductors with an electromagnetic source wave, which seems new even for bounded perfect conductors. 

 \subsection{Problem descriptions}

  It is supposed that
the incoming electromagnetic field is generated by the electric dipole
located at some source position $y\in\R^3$:
\be\label{electricdipole}
H^{in}(x,y,p):=\curl\,[p\Phi(x,y)],\quad
E^{in}(x,y,p):=\frac{i}{k}\curl H^{in}(x,y,p),\quad x\neq y,\quad
 \en
where $p\in\R^3$ is a constant vector representing the polarization
and $\Phi$
is the fundamental solution to the Helmholtz equation $\Delta u+k^2u=0$ with a
positive wave number $k$.
We will indicate the dependence of the scattered field and of the total field on the dipole point $y$ and the polarization $p$ by writing
$E^{sc}(x,y,p)$, $H^{s}(x,y,p)$ and $E(x,y,p)$, $H(x,y,p)$, respectively.
Let $\textbf{D}$ be a scatterer in $\R^3$ with Lipschitz boundary $\pa \textbf{D}$ such that the exterior
$\textbf{D}^{e}:=\R^3\ba\ov{\textbf{D}}$ of $\textbf{D}$ is connected.
The scatterer $\textbf{D}$ gives rise to a pair of scattered electromagnetic fields
$(E^{sc} , H^{sc} )\in H_{loc}(\curl, \textbf{D}^{e})\times H_{loc}(\curl, \textbf{D}^{e})$ which satisfies the
time-harmonic Maxwell equations
\be\label{EH_Maxwellequations}
\curl E^{sc} -ikH^{sc} =0,\quad \curl H^{sc} +ikE^{sc} =0\quad\mbox{in}\,\textbf{D}^{e}.
\en
For a perfect conductor, we have the perfectly conducting boundary condition, i.e.,
\be\label{pec}
\nu\times E=0\quad\mbox{on}\,\pa \textbf{D},
\en
where $\nu$ is the unit outward normal to $\pa \textbf{D}$ and $E:=E^{in} +E^{sc} $ is the total electric field.


In this section we shall consider two cases of $\textbf{D}$.

Case (i): $\textbf{D}$ is a compact subset in $\R^3$, which means a bounded perfect conductor.
In this case,
The scattered field $(E^{sc} , H^{sc} )$ is required to satisfy the Silver-M\"{u}ller radiation condition
\be\label{SMrc}
\lim_{r\rightarrow\infty}(H^{sc} \times x-r E^{sc} )=0
\en
where $r=|x|$ and the limit holds uniformly in all directions $\widehat{x}:=x/r$. In particular, the electric field has the asymptotic form
\ben
E^{sc} (x)=\frac{e^{ik|x|}}{|x|}\left\{ E^{sc} _\infty(\hat{x})+O(\frac{1}{|x|})    \right\},\qquad |x|\rightarrow\infty,
\enn
where $E^{sc} _\infty$ is known as the electric far-field pattern.
It is well known that there exists a unique solution $(E^{sc} , H^{sc} )\in H_{loc}(\curl, \textbf{D}^{e})\times H_{loc}(\curl, \textbf{D}^{e})$
of the scattering system (\ref{EH_Maxwellequations})-(\ref{SMrc}) (see e.g., \cite{CK2013,Monk}).
The inverse electromagnetic scattering problem we are interested in is to determine $\textbf{D}$ from  knowledge of the tangential components
$\nu\times E^{sc} $ of the electric far-field pattern $E^{sc} _\infty(\cdot, y, p)$ measured on $\mathbb{S}^2$.

Case (ii): $\textbf{D}$ is an unbounded scatterer but with a locally rough boundary $\pa \textbf{D}=:\Gamma$.
It is supposed that the boundary $\pa \textbf{D}$ is given by the graph of a bounded and
uniformly Lipschitz continuous function $x_3=f(\tilde{x})$ such that $f=0$ for $|\tilde{x}|>R$ for some $R>0$.
 Observing that the incoming wave fulfills the Silver-M\"{u}ller radiation condition (\ref{SMrc}), we require the scattered field to fulfill the half-space Silver-M\"{u}ller radiation condition in $x_3>0$, giving rise to the electric far-field pattern $E^{sc} _\infty(\hat{x})$ for all $\hat{x}\in \mathbb{S}^2_+$.
The well-posedness of this scattering problem (\ref{EH_Maxwellequations})-(\ref{pec}) 
has been established in \cite{LWZ} under the weaker  Silver-M\"{u}ller radiation condition of  integral type  for
 overfilled cavity scattering problems.
The inverse problem in this case is to determine $\textbf{D}$ from $E^{sc} _\infty(\hat{x})$ for all $\hat{x}\in \mathbb{S}^2_+$.

Applications of these two inverse scattering problems occur in such
diverse areas as medical imaging, nondestructive testing, radar, remote sensing, and geophysical exploration.
As in the acoustic case, one can easily see that these inverse problems are formally determined with all $x\in\G$, a single electric dipole located at $y\in \textbf{D}^e$ and a
fixed polarization $p\in\R^3$, since the measurements $E^{sc} _\infty(\hat{x})$ depend on the same number of variables as the boundary $\pa \textbf{D}$
to be reconstructed. 
The aim of this paper is to prove the unique determination of polyhedral-type perfect conductors $\textbf{D}$ by a single electric dipole with
a fixed wave number, a fixed dipole point and a fixed polarization.
The main tools we have used are the reflection principle for the Maxwell equations \cite{LiuYamamotoZou} and the path arguments developed
in \cite{HL2014}.

\subsection{Preliminary definition and reflection principle}\label{PolyRef}

In this subsection we introduce preliminary definitions and the reflection principle for the Maxwell's equations. We firstly clarify the definition of a general \emph{polyhedral scatterer}.
\begin{defn}\label{def:polyhedral}{\bf (Polyhedral scatterer)}
A subset $\textbf{D} \subset \R^3$ is called a \emph{polyhedral scatterer} if $\pa \textbf{D}$ is the union of
finitely many cells and the exterior $\textbf{D}^e:=\R^3\ba\ov{\textbf{D}}$ of $\textbf{D}$ is connected.
Here a cell is defined as the closure of an open connected subset of a two dimensional hyperplane.
\end{defn}
Note that the definition of a \emph{polyhedral scatterer} is more general than the terminology \emph{polyhedral obstacle} used in the literature.
A \emph{polyhedral obstacle} is defined as the union of finitely many convex polyhedra, which always coincides
with the closure of its interior. Hence, a \textit{polyhedral scatterer} we have defined is more general since it can be equivalently defined as
the union of a \textit{polyhedral obstacle} and finitely many cells.

\begin{defn}\label{def:perfectplane}{\bf (Perfect plane)}
Let $\Pi$ be a  two dimensional hyperplane in $\R^3$. A non-void open connected component $\mathcal {P}\subset\Pi$ will be called a {\em perfect plane}
of $E$ if $\nu\times E(\cdot, y, p)=0$ on $\mathcal {P}$.
\end{defn}

In the following, without loss of generality, we will always assume that a perfect plane $\mathcal {P}$ is meant to have been maximally connectedly
extended in $\textbf{D}^e\ba\{y\}$ since $E$ is analytic in any compact set in $\textbf{D}^e\ba\{y\}$.

\begin{defn}\label{def:perfectset}{\bf (Perfect set)}
$\mathcal {S}$ is called a {\em perfect set} of $E$ if
\ben
\mathcal {S}:=\{x\in \ov{\textbf{D}^e}: \mbox{there exists a perfect plane $\mathcal {P}$ of $E$ passing through $x$}.\}
\enn
\end{defn}

\begin{lem}\label{Lem:PSclose}
The perfect set $\mathcal {S}$ is closed, i.e., it contains all its limit points.
\end{lem}
\begin{proof}
Let $\{y_n\}_{n=1}^{\infty}$ be a sequence in $\mathcal {S}$ and $y^\ast\in \ov{\textbf{D}^e}$, such that $y_n\rightarrow\,y^\ast$ as $n\rightarrow\infty$.
Let $\mathcal {P}_n$ and $\Pi_n$ be the corresponding perfect plane and the hyperplane, respectively, passing through $y_n$ such that $\nu_n\times E=0$
on $\mathcal {P}_n$, where $\nu_n$ is the unit normal to $\Pi_n$.
By possibly choosing a subsequence, we may assume that $\nu_n\rightarrow\nu^\ast$ as $n\rightarrow\infty$.
We further assume that $\nu_n\cdot\nu^\ast\neq 0$.
Denote by $\Pi^\ast$ be the hyperplane passing through $y^\ast$ with unit normal $\nu^\ast$.

If $y^\ast\in\pa \textbf{D}$, then it is clear that $y^\ast\in \mathcal {S}$ by the definition of $\mathcal {S}$ and the perfect conducting boundary
condition (\ref{pec}) on $\pa \textbf{D}$. Thus, in the following, we may assume that $y^\ast\in \textbf{D}^e$. Taking a sufficiently small ball
$B_r(y^\ast)$ centered at $y^\ast$ with radius $r$ such that $\ov{B_r(y^\ast)}\subset \textbf{D}^e\ba\{y\}$.
Here, $y\in \textbf{D}^e$ is the dipole point of the electric dipole.
A direct consequence of this result is that $\mathcal {P}_n\cap B_r(y^\ast)=\Pi_n \cap B_r(y^\ast)$.
We will show that $\mathcal {P}^\ast:=B_r(y^\ast)\cap\Pi^\ast$ is a perfect plane of $E$, i.e.,
\be
\nu^\ast\times E=0\quad\mbox{on}\, \mathcal {P}^\ast.
\en
To do this, for all $x^\ast\in \mathcal {P}^\ast$, let $l^\ast$ be the straight line passing through $x^\ast$ with direction $\nu^\ast$.
Note that for all $x\in l^\ast$, we have $x=x^\ast+t\nu^\ast$ for some $t\in\R$. While for any $x\in\Pi_n$, we have $(x-y_n)\cdot\nu_n=0$.
Since $\nu_n\cdot\nu^\ast\neq 0$, by straightforward calculations, we found that the straight line $l^\ast$ intersects with each hyperplane $\Pi_n$
and the intersection points $x_n:=l^\ast\cap \Pi_n$ are
\ben
x_n=x^\ast+t_n\nu^\ast\;\;\mbox{with}\;\;t_n=\frac{(y_n-x^\ast)\cdot\nu_n}{\nu_n\cdot\nu^\ast}, \,n=1,2,\cdots.
\enn
Then
\ben
\lim_{n\rightarrow\infty}|x_n-x^\ast|
=\lim_{n\rightarrow\infty}|t_n|
&=&\lim_{n\rightarrow\infty}\left|\frac{(y_n-x^\ast)\cdot\nu_n}{\nu_n\cdot\nu^\ast}\right|\cr
&\leq& \lim_{n\rightarrow\infty}\left|\frac{(y_n-y^\ast)\cdot\nu_n}{\nu_n\cdot\nu^\ast}\right|
    +\lim_{n\rightarrow\infty}\left|\frac{(y^\ast-x^\ast)\cdot\nu_n}{\nu_n\cdot\nu^\ast}\right|\cr
&=&0,
\enn
where we have used the fact that $y_n\rightarrow y^\ast$ and $(y^\ast-x^\ast)\cdot\nu^\ast=0$.
This means $x_n\rightarrow x^\ast$. By possibly choosing a subsequence, we may assume that $x_n\in B_r(y^\ast)$.
This implies that
\ben
\nu^\ast\times E(x^\ast)=\lim_{n\rightarrow\infty}\nu_n\times E(x_n)=0
\enn
because of the fact that $x_n\in\mathcal {P}_n$.
The proof is complete.
\end{proof}

Denote by $\mathcal{R}_{\Pi}$ the reflection with respect to a hyperplane $\Pi$ in $\R^3$.
Now we are ready to state the reflection principle for the Maxwell's equations \cite{LiuYamamotoZou}.

\begin{lem}\label{Lem:RP}
{\bf(Reflection principle for the Maxwell's equations w.r.t. perfect planes)}
Suppose that $\Omega\subset \R^3$ is a
symmetric connected domain with respect to a two dimensional hyperplane $\Pi$ and
that $\La:=\Omega\cap \Pi\neq\emptyset$.
Denote by $\Om^+$ and $\Om^-$ the two connected subdomains of $\Om$ separated by $\La$.
If $(E, H)$ solves (\ref{EH_Maxwellequations}) in $\Om$ and $\nu\times E=0$ on $\La$,
then $(E, H)$ satisfies
\be\label{odd}
E(x)+\mathcal{R}_{\Pi_0}E(\mathcal{R}_{\Pi}(x))=0,\quad H(x)-\mathcal{R}_{\Pi_0}H(\mathcal{R}_{\Pi}(x))=0,\qquad x\in \Om,
\en
where $\Pi_0$ is the hyperplane passes through the original point $O$ and is parallel to $\Pi$.
\end{lem}

The reflection principle described in Lemma \ref{Lem:RP} immediately gives us the following properties of solutions to the Maxwell's equations.
\begin{cor}\label{cor:1} With the notations used in Lemma \ref{Lem:RP}, we suppose that $(E, H)$ solves (\ref{EH_Maxwellequations}) in $\Om$
and $\nu\times E=0$ on $\La$.
\begin{description}
\item[(i)]
If $\La_0$ is a perfect plane of $E$ in $\Om^+$, then $\mathcal{R}_{\Pi}(\La_0)\subset\Om^-$ is also a perfect plane of $E$.
\item[(ii)] If $E$ is singular at  $y\in \Om^+$, then $E$ is also singular at  $y^{\ast}:=\mathcal{R}_\Pi(y)\in \Om^-$.
\end{description}
\end{cor}
 From the second assertion of Corollary \ref{cor:1}, we see the number of singularities of $E$ in $\Om$ must be even. This fact will be
 utilized to justify the uniqueness within polyhedral scatterers, since the total field has only one singular point (i.e., the dipole point of
 the electric dipole) in the exterior $\textbf{D}^e$ of the scatterer $\textbf{D}$ under investigation.
\begin{rem}\label{remark 3.1}
From Lemma \ref{Lem:RP}, one can see that the reflected electric field takes the form
\ben
E^{re}(x)=-\mathcal{R}_{\Pi}E^{in}(\mathcal{R}_{\Pi}(x))\qquad \mbox{with}\quad\Pi:=\{x_3=0\}
\enn
where $E^{in}$ is the electric dipole given in (\ref{electricdipole}). If $E^{in}=pe^{ikx\cdot d}$ with $p\bot d$ and $|d|=1$, one can prove the existence and uniqueness of the Silver-M\"{u}ller radiation solution $E-E^{in}-E^{re}$; see \cite{LWZ}.
\end{rem}

\subsection{Uniqueness with a single incoming wave}\label{PolyRef}

In this section, we will present two uniqueness results in inverse electromagnetic scattering by general polyhedral-type scatterer in $\R^3$. The basic
ideal is from our previous work on the acoustic case \cite{HL2014}.

\begin{thm}\label{mainresultbounded}
Let $\textbf{D}$ be a bounded perfect polyhedral scatterer.
Then, for a fixed wave number $k>0$, the boundary $\pa \textbf{D}$ can be
uniquely determined by the electric far-field pattern  $E_\infty^{sc} (\hat{x}, y, p)$ for all $\hat{x}\in\mathbb{S}^2$ generated by a single electric dipole $E^{in} (\cdot, y, p)$ located at
a fixed dipole point $y\in \textbf{D}^e$ and a fixed polarization $p\in\R^3$.
\end{thm}
\begin{proof}
Assume that two bounded perfect polyhedral scatterers $D_1$ and $D_2$ generate the  same electric far-field pattern
$E_{\infty,1}^{sc} (\hat{x}, y, p)=E_{\infty,2}^{sc} (\hat{x}, y, p) $ on $\mathbb{S}^2$ due to
an electric dipole $E^{in} (\cdot, y, p)$ located at $y\in \Om_0$ with the polarization $p\in\R^3$.
Here, $\Om_0$ denotes the unbounded connected component of $D_1^e \cap D_2^e$. We are aimed at proving $\partial D_1=\partial D_2$.
By \cite[Theorem 6.10]{CK2013} and the unique continuation of solutions to the
Maxwell's equations, we see
\be\label{eq}
 E_1(\cdot, y, p)=E_2(\cdot, y, p) \quad\mbox{in}\quad \Omega_0\backslash\{y\}.
\en
 If $\partial D_1\neq \partial D_2$, without loss of generality we may always assume there exists a {\em perfect plane} $\mathcal {P}_0$ of $E_1$ in $D_1^e$.
 This follows from the relation (\ref{eq}) together with  the fact that $D_1$ and $D_2$ are both polyhedral scatterers in the sense of
 Definition \ref{def:polyhedral} and that $D_j^e$ for $j=1,2$  are connected.
Next, we shall carry out the proof by deriving a contraction.  For clarity we divide our proof into three steps.

\emph{\textbf{Step 1: Path argument.}}
Let $\mathcal {S}_1$ be the perfect set of $E_1$ as defined in Definition \ref{def:perfectset}.
The perfect set $\mathcal {S}_1\neq\emptyset$, because $\mathcal {P}_0\subset \mathcal {S}_1$.
Choose a point $y_0\in \mathcal {P}_0$  and a continuous injective curve
$\g(t)$ for $t\geq 0$ connecting $y_0$ and the dipole point $y$ of the electric dipole. Without loss of generality, we assume that
$\gamma(0)=y_0$ and $\gamma(T)=y$ for some $T>0$. Let $\mathcal{M}$ be the set of intersection points of $\gamma$ with $\mathcal {S}_1$, i.e.,
\ben
\mathcal{M}=\{y_n: \mbox{there exist a perfect plane $\mathcal {P}_n \subset \mathcal {S}_1$ and $t_n\geq0$ such that
$\mathcal {P}_n\cap \gamma(t_n)=y_n$ } \}.
\enn
It is clear that the points contained in $\mathcal{M}$ are uniformly bounded, since $\gamma(t)$ is a bounded curve with finite length.
By Lemma \ref{Lem:PSclose}, we know $\mathcal {S}_1$ is closed. This implies that $\mathcal{M}$ is closed since $\gamma(t)$ is also
closed. Hence, $\mathcal{M}$ is compact, and we can find some $t^*>0$
such that there exists a perfect plane $\mathcal {P}^*$ of $E_1$ intersecting with $\gamma(t)$ at $t=t^*$ and that
\ben
 \gamma(t)\cap \mathcal{M}=\emptyset,  \qquad\forall\quad T> t>t^*.
\enn
Denote by $\nu^\ast$ be the unit normal to $\mathcal {P}^*$
 Note that $t^* <T$, because $\nu^\ast\times E_1$ vanishes at $\gamma(t^*)$ but is singular at $\gamma(T)=y$.

 \emph{\textbf{Step 2: Reflection argument.}}
 Let $\Pi^*$ be the hyperplane containing $\mathcal {P}^*$.
 We now apply  Corollary \ref{cor:1} to prove the existence of a symmetric open set $\Omega\supset\mathcal {P}^*$ with respect to  $\Pi^*$
 such that $\Omega\subset D_1^e$ and $y\in \Omega$. This will be done in the following paragraph.

 Choose $x^+=\gamma(t^*+\epsilon)$ for $\epsilon>0$ sufficiently small such that
 $t^*+\epsilon<t<T$ and define $x^-:=\mathcal{R}_{\Pi^*}(x^+)$. Let $G^\pm$ be the connected component of
 $\R^3\ba\{\overline{D}_1\cup \mathcal {P}^*\}$ containing $x^\pm$, and denote by $\Omega^\pm$ the connected component of
 $G^\pm\cap \mathcal{R}_{\Pi^*}(G^\mp)$ containing $x^\pm$. Setting $\Omega:=\Om^+\cup\mathcal {P}^*\cup \Om^-$, we observe that
 $\Omega\subset D_1^e$ is a connected symmetric domain with respect to $\Pi^*$ whose boundary is a subset of
 $(\mathcal {S}_1\cup\pa D_1)\cup \mathcal{R}_{\Pi^*}(\mathcal {S}_1\cup\pa D_1)$.
 Thus, by Corollary \ref{cor:1} $({\rm i})$, $\nu\times E_1=0$ on $\pa\Om$.
 It now remains to prove $y\in\Om$.
 Assume to the contrary that $y\notin\Om$. Since $x^+=\gamma(t^*+\epsilon)\in \Omega^+$ and the continuous curve $\gamma(t)$ for
 $t^*+\epsilon<t<T$ lies in $D_1^e$,  $\gamma(t)$ must intersect $\partial \Omega\cup\mathcal{R}_{\Pi^*}(\partial D_1) $ at
 some $t^{**}>t^*+\epsilon$.
This implies the existence of a new perfect plane intersecting $\gamma(t)$ at $t^{**}>t^*$, contradicting the obtained perfect plane
$\mathcal {P}^*$ at $t=t^*$. Hence $y\in\Om$.

 \emph{\textbf{Step 3: End of the proof.}}
Let $\Om$ and $\Pi^*$ be given as in Step 2.
We observe that $y^*:=\mathcal{R}_{\Pi^*}(y)\in \Om$,  since $y\in\Om$ and
$\Om$ is a connected symmetric domain with respect to $\Pi^*$.
By Corollary \ref{cor:1} $({\rm ii})$, $E_1$ is also singular at $y^*\,(\neq y)$. However,
this is a contradiction to the analyticity of $E_1$ in $\Om\subset D_1^e\ba\{y\}$.
The proof is complete.
\end{proof}

\begin{rem}\label{remark 3.2} To the best of our knowledge, it is still an open problem how to uniquely determine the shape of a bounded perfect conductor with a single incoming wave (for example, a plane wave or an electric dipole). It was proved in \cite{HL2014} that a sound-soft ball can be uniquely determined by an acoustically point source wave. However, it remains unclear to us the unique determination of perfectly conducting ball with an electric dipole.
\end{rem}

For electromagnetic scattering from a locally perturbed rough surface, we have the analogous uniqueness result to Theorems \ref{polypoint} and \ref{polyuni} as follows.

\begin{thm}\label{mainresultunbounded}
Let $\textbf{D}$ be an unbounded perfect polyhedral scatterer with locally rough boundary $\pa \textbf{D}$.
Then, for a fixed wave number $k>0$, the boundary $\pa \textbf{D}$ can be
uniquely determined by the electric far-field data  $ E^{sc}_\infty (\hat{x}, y, p)$ for all $\hat{x}\in \mathbb{S}^2_+$ generated by a single electric dipole $E^{in} (\cdot, y, p)$ located at
a fixed dipole point $y\in \textbf{D}^e$ and a fixed polarization $p\in\R^3$.
\end{thm}
\begin{proof}
Using the reflection principle of Lemma \ref{Lem:RP}, one can prove the one-to-one correspondence between the electric far-field data $ E^{sc}_\infty (\hat{x}, y, p)$ for all $\hat{x}\in \mathbb{S}^2_+$ and the electric field $E^{sc}(x)$ for $x\in \textbf{D}^e\backslash\{y\}$; see Theorem \ref{farscat} in the acoustic case. Repeating the reflection and path arguments in the proof of Theorem \ref{mainresultbounded}, one can verify Theorem \ref{mainresultunbounded} in the same manner.
Note that although the surface $\partial \textbf{D}$ is unbounded, the reflection and path arguments for bounded polyhedral conductors are still applicable, because the unperturbed ground plane $\{x_3=0\}$ is supposed to be known in advance.
\end{proof}

\section*{Acknowledgement}
The work of G. Hu is supported by the NSFC grant (No. 11671028) and NSAF grant (No. U1530401). The work of B. Yan is supported by the National Natural Science Foundation of China (No. 61603226) and the Fund of Natural Science of Shandong Province (No. ZR2018MA022). The authors would like to thank Xiaodong Liu (CAS, China) for reading through the original version of this manuscript and for his comments and generous discussions which help improve the presentation.

\end{document}